\documentclass[12pt]{amsart}
\usepackage{amscd,amsmath}

\newcommand{\la}{\lambda}

\newcommand{\dif}{\mathrm{d}}

\newcommand{\norm}[1]{\Vert #1\Vert}

\newcommand{\ov}{\overline}

\newcommand{\Ee}{\mathcal{E}}
\newcommand{\VV}{\mathcal{V}}
\newcommand{\Hh}{\mathcal{H}}
\newcommand{\Cl}{\mathrm{Cl}}
\newcommand{\cl}{\mathrm{cl}}

\newcommand{\Ker}{\mathop\mathrm{Ker\,}}

\newcommand{\CC}{\mathbb{C}}

\newcommand{\RR}{\mathbb{R}}
\newcommand{\ZZ}{\mathbb{Z}}

\newcommand{\SQ}{\mathcal{S}}

\theoremstyle{plain}
\newtheorem{te}{Theorem}

\newtheorem{co}{Corollary}
\newtheorem{lm}{Lemma}

\theoremstyle{definition}
\newtheorem{de}{Definition}
\newtheorem{re}{Remark}
\newtheorem{ex}{Example}

\begin{document}

\title{A characterization of Dirac morphisms}

\author{E. Loubeau}
\address{D\'epartement de Math\'ematiques\\
Universit\'e de Bretagne Occidentale\\
6, Avenue Victor Le Gorgeu\\
CS 93837, 29238 Brest Cedex 3\\
France}
\email{loubeau@univ-brest.fr}

\author{R. Slobodeanu}
\address{Faculty of Physics, Bucharest University,
405 Atomi\c stilor Str., CP Mg-11, RO - 077125 Bucharest, Romania.}

\email{radualexandru.slobodeanu@g.unibuc.ro}

\date{\today}

\thanks{The second author benefited from a one-year fellowship of the Conseil G{\'e}n{\'e}ral du Finist{\`e}re}

\begin{abstract}
Relating the Dirac operators on the total space and on the base manifold of a horizontally conformal submersion, we
characterize Dirac morphisms, i.e. maps which pull back (local) harmonic spinor fields onto (local) harmonic spinor fields.
\end{abstract}

\maketitle

\section{Introduction}

Introduced by Jacobi~\cite{Jacobi} in 1848, harmonic morphisms are maps which pull back local harmonic functions onto
harmonic functions and, more recently, they were characterized by Fuglede~\cite{Fuglede} and Ishihara~\cite{Ishihara} as horizontally weakly
conformal harmonic maps. Their dual nature of analytical and geometrical objects has led to a rich theory (cf.~\cite{B-W}) which has encouraged the study of various other morphisms, that is maps preserving germs of certain differential operators.
The central role of the Dirac operator in differential geometry and mathematical physics called for this approach to be applied to
harmonic spinors. Unlike previous cases, the first hurdle is to make sense of a notion of pull-back of spinors by a map. This requires the identification of the spinor bundles involved, necessarily restricting our investigation to horizontally conformal maps between Riemannian manifolds and even-dimensional targets (cf. Section~\ref{section2}). Combining a chain rule for the Dirac operator and a local existence lemma, we show that a horizontally conformal submersion between spin manifolds is a Dirac morphism if and only if its horizontal distribution is integrable and the mean curvature of the fibres is related to the dilation factor, in a manner reminiscent of the fundamental equation for harmonic morphisms.
We conclude with some simple examples between Euclidean spaces and explicit our results in the set-up of~\cite{moro}, which inspired
initially our construction.

\section{Pull-back of a spinor}~\label{section2}

Let $(M^{m}, g)$ be a spin Riemannian manifold, the two-sheeted covering $\mathrm{Spin}(m) \stackrel{\rho}{\longrightarrow}
\mathrm{SO(m)}$ induces a double cover $\chi: P_{\mathrm{Spin}(m)}M \longrightarrow P_{\mathrm{SO(m)}}M$ of the bundle of positively
oriented orthonormal frames by the principal $\mathrm{Spin}(m)$-bundle over $M$, such that $\chi(s \cdot g)=\chi(s)\cdot \rho(g)$,
$\forall  s \in P_{\mathrm{Spin}(m)}M, g \in \mathrm{Spin}(m)$. The associated bundle $\Cl(M) = P_{\mathrm{SO(m)}}M \times _{cl_{m}}
\Cl_{m}$ is the {\em Clifford bundle}, where $\Cl_{m}$ is the Clifford algebra and $\cl_{m}$ the representation of $\mathrm{SO(m)}$ into
$\mathrm{Aut}(\Cl(\RR^m))$, and the {\em spinor bundle} is $\SQ M = P_{\mathrm{Spin}(m)}M \times _{\gamma} \SQ_{m}$, with $\gamma$ the
spinorial representation of $\mathrm{Spin}(m)$ on the Clifford module $\SQ_{m}=\CC^{2^{[m/2]}}$ (cf.~\cite {law}).

A  {\em spinor field} is a (smooth) section of $\SQ M$, $\Psi: U \subset M \longrightarrow \SQ M$, $\Psi(x)=[s_{x}, \psi(x)]$,
where $s_{x} \in P_{\mathrm{Spin}(m)}M$ is a spinorial frame at $x \in M$ and
$\psi: U \longrightarrow \SQ_{m}$, the equivalence class being defined by
$$[s, \psi]=[s \cdot g^{-1}, \gamma(g) \psi],$$
for all $g \in \mathrm{Spin}(m)$. The covariant derivative is
$$\nabla_{e_{j}} \Psi = \left[s, \ \ \dif \psi (e_{j}) + \tfrac{1}{2}\sum_{k<l=1}^{m} \Gamma_{jk}^{l}\ e_{k} \cdot e_{l} \cdot \psi \right],$$
where $\chi(s)=\{e_{1},...,e_{m}\}$ is an orthonormal frame of $TM$, ``$\cdot$'' is the Clifford multiplication
and the symbols $\{\Gamma_{ij}^{k}\}_{i,j,k=1,\dots,m}$ are defined by $\nabla_{e_{i}}e_{j}=\Gamma_{ij}^{k}e_{k}$.

The {\em Dirac operator} is the first-order differential operator defined by
\begin{align*}
D^M:\Gamma (\SQ M) &\to \Gamma (\SQ M) \\
 \Psi &\mapsto D^M \Psi = \sum_{j=1}^{m} e_{j} \cdot \nabla_{e_{j}} \Psi .
\end{align*}

The above constructions extend to any oriented Riemannian bundle and, as for
orientation, given three vector bundles $\Ee^{\prime}$, $\Ee^{\prime \prime}$ and
$\Ee = \Ee^{\prime} \oplus \Ee^{\prime \prime}$ over $M$, a choice of spin-structure
on any two of them uniquely determines a spin structure on the third (\cite{law}).

\begin{de}
A smooth map $\pi : (M^m , g) \to (N^n ,h)$ between Riemannian manifolds is a {\em horizontally conformal map} if,
at any point $x\in M$, $d\pi_{x}$ maps the {\em horizontal space} $\Hh_{x} = (\ker d\pi_{x})^{\perp}$ conformally onto
$T_{\pi(x)}N$, i.e. $d\pi_{x}$ is surjective and there exists a number $\lambda (x) \neq 0$ such that
$$(\pi^{*}h)_{x}\Big|_{\Hh_{x} \times \Hh_{x}} = \lambda^{2}(x) g_{x}\Big|_{\Hh_{x} \times \Hh_{x}}.$$
The function $\lambda$ is the {\em dilation} of $\pi$ and the orthogonal complement of $\Hh_{x}$ is the {\em vertical distribution}
$\VV_{x} = \ker d\pi_{x}$.\\
The mean curvatures of the distributions $\Hh$ and $\VV$ are denoted $\mu^{\Hh}$ and $\mu^{\VV}$ and $I^{\Hh}$
is the integrability tensor of $\Hh$.\\
A frame $\{V_{a},X_{i}^{*}\}^{a=1,\dots,m-n}_{i=1,\dots,n}$ of $TM$ will be called {\em adapted} if $V_{a} \in \VV, a =
1,\dots,m-n$
and $\{X_{i}^{*}\}_{i=1,\dots,n}$ is the horizontal lift by $\pi$ of an orthonormal frame $\{X_{i}\}_{i=1,\dots,n}$ on $N$. \\
Note that $\lambda \equiv 1$ corresponds to Riemannian submersions.\\
We call the map $\pi : (M^m , g_{1}) \to (N^n ,h)$, where $g_{1}=\pi^{*}h+g^{\VV}$, the {\em associated} Riemannian submersion of $\pi : (M^m , g) \to (N^n ,h)$.
\end{de}

For the remainder of the article, we will make the blanket assumption that the dimension $n$ of the manifold $N$ is even.
Since a general submersion $\pi : (M^m , g) \to (N^n ,h)$, between spin Riemannian manifolds, splits the tangent bundle
$TM$ into $\Hh \oplus \VV$, if $\Hh$ admits a spin structure, so does $\VV$ and
\begin{equation} \label{splitc}
\Cl(M) \ = \ \Cl(\Hh) \ \hat{\otimes} \ \Cl(\VV) .
\end{equation}
The spin structures $P_{\mathrm{Spin}(n)}\Hh$ and $P_{\mathrm{Spin}(m-n)}\VV$ induce
a spin structure $P_{\mathrm{Spin}(m)}M$ by prolongation of the principal bundle
$P_{\mathrm{Spin}(n) \times \mathrm{Spin}(m-n)}M$ (cf.~\cite{hus}). General
properties of associated bundles of reduced principal bundles (\cite[Theorem
3.1]{hus}) together with a dimension count yield the following isomorphisms of
(associated) vector bundles
\begin{align} \label{splits2}
\SQ M \ &= \ \SQ \Hh \otimes \SQ \VV .
\end{align}

For any map $\pi : M \to N$ into a spin manifold, consider the pull-back spinor bundle
$$\pi^{-1}\SQ N = \{ (x, [s, \psi]) \in M \times \SQ N \ | \ \varpi_{\SQ}([s, \psi])=\pi(x) \},$$
where $\varpi_{\SQ}$ is the projection map of $\SQ N$.\\
If $\pi$ is a Riemannian submersion, then the isomorphism $\pi^{-1}\SQ N  =  \SQ
\Hh$, due to the identification of orthonormal frames, simplifies \eqref{splits2} into (cf. \cite{bis})
\begin{equation}
 \SQ M =\pi^{-1}\SQ N \otimes \SQ \VV.
\end{equation}

\begin{re}
When $\pi$ is a Riemannian submersion with totally geodesic fibres, $\Hh$ 
complete and $N$ connected then the fibres are isometric to a Riemannian manifold $F$. If
$N$ and $F$ are spin manifolds, consider on $M$ the induced spin structure and, via
the isomorphism $\pi^{-1}\SQ N  =  \SQ \Hh$, \eqref{splits2} reads (see \cite{moro})
$$ \SQ M =\pi^{-1}\SQ N \otimes \SQ F.$$
\end{re}

\begin{re}
Since $n$ is even, the Clifford algebra $\Cl_n$ possesses an irreducible complex module
$\mathcal{S}_{n}$ of complex dimension $2^{n/2}$, the complex spinor module. When
restricted to $\Cl^{0}_{n}$ the spinor module decomposes into $\mathcal{S}_{n} =
\mathcal{S}^{+}_{n}\oplus \mathcal{S}^{-}_{n}$, the submodules of spinors of positive and  negative chirality, 
characterized by the action of the volume element, once an
orientation is given.
In particular, the spin group $\mathrm{Spin}(n) \subset \Cl^{0}_{n}$ acts on
$\mathcal{S}^{+}_{n}$ and on $\mathcal{S}^{-}_{n}$ (the spinor representations).\\
Moreover, $\mathcal{S}_{n+1}$ pulls back to $\mathcal{S}_{n}$ under the
algebra isomorphism $\Cl_{n} = \Cl^{0}_{n+1}$. In other words, we can regard
$\mathcal{S}_{n+1}$ as the spinor representation of $\Cl_{n}$, provided we define the
action of $\Cl_{n}$ on $\mathcal{S}_{n+1}$ by
$v \otimes \sigma \rightarrow e_0 \cdot v \cdot \sigma$.
\end{re}

When $\pi : (M^{n+1},g) \to (N^{n},h)$ is a Riemannian submersion with
one-dimensional fibres into a spin manifold $N^{n}$, the manifold $M^{n+1}$ inherits
a natural spin structure, cf~\cite{amoro}. Moreover, $\SQ M = \pi^{-1}\SQ N$.

These identifications justify the following definition.
\begin{de}[Riemannian submersions]
Let $\pi : (M^m , g) \to (N^n ,h)$ be a Riemannian submersion between spin manifolds
and endow the vertical bundle with the induced spin structure (if $m=n+1$, we
consider on $M$ the natural spin structure inherited from $N$). Let $\Psi=[s, \psi]$
be a (local) spinor field on $N$. Since a local spin frame
$s=\{X_{i}\}_{i=1,\dots,n}$ on $N$ lifts to an adapted spin frame on $M$
$$\tilde{s}=\{X_{i}^{*}, V_{a}\}_{i=1,\dots,n}^{a=1,\dots,m-n}\in
P_{\mathrm{Spin}(n) \times _{\mathbb{Z}_{2}} \mathrm{Spin}(m-n)}M \vert _{\pi^{-1}U},$$
where $X_{i}^{*}$ is the horizontal lift of $X_{i}$ and $\{V_{a}\}_{a=1,\dots,m-n}$ is an orthonormal frame of $\VV$,
we define the {\em pull-back} $\widetilde \Psi$ of $\Psi$ to be
\begin{itemize}
\item If $m-1=n$, the section $\widetilde \Psi=[\Tilde{s}, \psi \circ \pi]$ of the bundle
$\pi^{-1}\SQ N$, identified with $\SQ M$.\\
\item If $m-n \geq 2$, the section
$\widetilde{\Psi}=[\widetilde{s}, \widetilde \psi=(\psi \circ \pi) \otimes \alpha]$
in $\pi^{-1}\SQ N \otimes \SQ \VV$, identified with $\SQ M$, where $\alpha$ is a fixed (non-zero) section of $\SQ \VV$.
\end{itemize}
\end{de}

\begin{re}
When $m-n \geq 2$, this notion of pull-back depends on the choice of the section $\alpha$. In general, there exists no such non-vanishing global section.
\end{re}

\begin{re}
Note that Clifford multiplication with this kind of spinor fields is given by
$$
X^* \cdot \widetilde{\psi}  = \widetilde{X \cdot \psi}; \quad V \cdot
\widetilde{\psi} = \mathrm{i} \widetilde{\ov{\psi}},
$$
when $m=n+1$, and by
$$
X^* \cdot ((\psi\circ\pi) \otimes \alpha) = (( X \cdot \psi)\circ\pi) \otimes \alpha;
\quad V \cdot ((\psi\circ\pi) \otimes \alpha) = (\ov \psi\circ\pi) \otimes V \cdot
\alpha,
$$
when $m-n \geq 2$, where $X^*$ is the horizontal lift of the vector field $X \in
\Gamma(TN)$, $V$ is a unit vertical vector field and $\ov \psi$ is the conjugate of $\psi$ with respect to the $\ZZ
_2$-graduation (see~\cite{moro}).
\end{re}

\begin{re}
For a horizontally conformal submersion $\pi : (M^m , g) \to (N^n ,h)$, we deform the metric on $M$ into
$g_{1}=\pi^{*}h+g^{\VV}$ and denote by (cf.~\cite{law})
\begin{align*}
& \xi_{\gamma} : \SQ M_{1} \to \SQ M \\
& \xi_{\gamma}([s, \psi]) = [\xi(s), \psi],
\end{align*}
where  $\xi(s)=\{ E_i=\lambda E_i ^1, V_a \}$ if $s=\{ E_i ^1, V_a \}$, the bundle isometry induced by the Spin-equivariant map
$$\xi: P_{\mathrm{Spin}(n) \times _{\mathbb{Z}_{2}} \mathrm{Spin}(m-n)}M_{1}
\to P_{\mathrm{Spin}(n) \times _{\mathbb{Z}_{2}} \mathrm{Spin}(m-n)}M$$
given by the natural correspondence between adapted orthogonal frames with respect to the two metrics:
$E_{i}^{1}=\lambda^{-1} E_{i}, V^{1}_{a}=V_{a}$. \\
The Clifford multiplication will be given by
$$
E_i \cdot \Psi = \xi_{\gamma}\left(E_i ^1 \cdot \Psi_1\right), \qquad
V_a \cdot \Psi = \xi_{\gamma}\left(V_a \cdot \Psi_1\right),
$$
where $\Psi=\xi_{\gamma} \circ \Psi_{1}$.
\end{re}

\begin{de}[Horizontally conformal submersions]
Let $\pi : (M^m , g) \to (N^n ,h)$ be a horizontally conformal submersion between spin manifolds and
endow the vertical bundle with the induced spin structure. Let $\Psi=[s, \psi]$ be a (local) spinor field on $N$.
The {\em pull-back} of $\Psi$ is
$\widetilde{\Psi}=\xi_{\gamma} \circ \widetilde{\Psi}_{1}$, where
$\widetilde{\Psi}_{1}$ is the pull-back of $\Psi$ by the associated Riemannian submersion $\pi : (M^m , g_{1}) \to (N^n ,h)$
and $\xi_{\gamma}$ the bundle isometry between $\SQ M_{1}$ and $\SQ M$.
\end{de}

\section{Dirac morphisms with high dimensional fibres}

Throughout this section $\pi$ has fibres of dimension at least two.

\begin{de}\label{def4}
A horizontally conformal submersion $\pi : (M^m ,g) \to (N^n ,h)$ between spin manifolds is called a {\em Dirac morphism} if
there exists a section $\alpha \in \Gamma(\SQ \VV)$, $\nabla^{\VV}$-parallel in horizontal directions with $D^{\VV}\alpha - \frac{n}{2} \mu^{\Hh} \cdot \alpha= 0$, and if for any local harmonic spinor $\Psi$ defined on $U\subseteq N$, such that $\pi^{-1}(U)\neq \emptyset$, 
the pull-back of $\Psi$ (with respect to $\alpha$) $\widetilde{\Psi}=\xi_{\gamma} \circ \widetilde{\Psi}_{1}$ is a harmonic spinor on $\pi^{-1}(U) \subseteq M$, where $\widetilde{\Psi}_{1}$ is the pull-back (with respect to $\alpha$) of $\Psi$ by the associated Riemannian submersion.
\end{de}

\begin{re}
We assume, in Definition~\ref{def4}, the existence of the section $\alpha$ in order to construct pull-backs of spinors. Though these pull-backs will depend on the choice of $\alpha$ (if any), the two conditions on $\alpha$ will make the notion of Dirac morphism independent of the choice of such a section $\alpha$.
\end{re}

We first need some lemmas.

\begin{lm}[Chain rule]
Let $\pi : (M^m ,g) \to (N^n ,h)$ be a horizontally conformal submersion of dilation $\lambda$ ($m-n \geq 2$) and $\psi$ a (local) spinor field on $N$. 
If $ \widetilde{\Psi}$
is the pull-back of $\Psi$ by $\pi$, with respect to some section $\alpha \in \Gamma(\SQ \VV)$, then

\begin{equation} \label{chainhwc}
\begin{split}
D^{M} \widetilde{\psi}  =
&\lambda \widetilde{D^{N}\psi}
-\tfrac{1}{2} \left[(m-n) \mu^{\VV}+(n-1)\mathrm{grad}^{\Hh}(\mathrm{ln}\lambda) \right] \cdot \tilde \psi  \\
&+ \sum _{i=1}^{n} E_{i} \cdot (\psi \circ \pi) \otimes \nabla^{\VV}_{E_{i}}\alpha
+\tfrac{1}{4} \ I^{\Hh} \cdot \tilde \psi  \\
&+(\ov \psi \circ \pi) \otimes \left[D^{\VV}\alpha-\tfrac{n}{2} \mu^{\Hh} \cdot \alpha \right],
\end{split}
\end{equation}
where  $\{E_{i}\}_{i=1\dots,n}$ is a local orthonormal horizontal frame on $M$ and $I^{\Hh} \cdot \widetilde \psi$ denotes
$\sum _{i < j = 1}^{n} E_{i} \cdot E_{j} \cdot I^{\Hh}(E_{i} , E_{j}) \cdot \widetilde \psi$
(the standard action of vector-valued 2-forms on spinor fields).
\end{lm}

\begin{proof}
Let $\pi$ be a horizontally conformal submersion of dilation $\lambda$. Let $\{X_i \}_{i=1,\dots,n}$ be an orthonormal frame on $(N,h)$ and
$\{V_a, X^{*}_i \}_{i=1,\dots,n}^{a=1,\dots,m-n}$ an orthonormal adapted frame on $(M,g_1)$, where $g_1=\pi^* h+g^{\VV}$. With respect to the metric $g$, $\{V_a, \lambda  X^{*}_i \}_{i=1,\dots,n}^{a=1,\dots,m-n}$ is an orthonormal adapted frame. Denote by $\nabla$ and $\nabla ^1$ the (spinorial) connections corresponding to $g$ and $g_1$, and 
note $E_i ^1 = X_i ^*$, $E_i = \lambda X_i ^*$.

As $D^{M} \widetilde{\Psi} =[\tilde s, D^{M} \tilde \psi]$, for the pull-back spinor field $\tilde \psi$

\begin{align*}
D^{M} \widetilde{\psi} & =  \sum _{i=1}^{n} E_{i} \cdot \nabla_{E_{i}} \widetilde{\psi}+
 \sum _{a=1}^{m-n} V_{a} \cdot \nabla_{V_{a}} \widetilde{\psi}  \\
&= \sum _{i=1}^{n} E_{i} \cdot E_{i}((\psi \circ \pi) \otimes \alpha) \qquad (H_{0}) \\
&+ \tfrac{1}{4}\sum _{i, j, k=1}^{n} E_{i} \cdot g(\nabla_{E_{i}}E_{j}, E_{k}) \ E_{j} \cdot E_{k} \cdot \widetilde{\psi} \qquad (H_{1}) \\
&+ \tfrac{1}{2}\sum _{i, j, a=1}^{n,m-n} E_{i} \cdot g(\nabla_{E_{i}}E_{j}, V_{a}) \ E_{j} \cdot V_{a} \cdot \widetilde{\psi} \qquad (H_{2}) \\
&+ \tfrac{1}{4}\sum _{i, a, b=1}^{n,m-n} E_{i} \cdot g(\nabla_{E_{i}}V_{a}, V_{b}) \ V_{a} \cdot V_{b} \cdot \widetilde{\psi} \qquad (H_{3})
\end{align*}

\begin{align*}
\quad & +\sum _{a=1}^{m-n} V_{a} \cdot  V_{a}((\psi \circ \pi) \otimes \alpha) \qquad (V_{0}) \\
&+ \tfrac{1}{4}\sum _{i, j , a=1}^{n,m-n} V_{a} \cdot g(\nabla_{V_{a}}E_{i}, E_{j}) \ E_{i} \cdot E_{j} \cdot \widetilde{\psi} \qquad (V_{1}) \\
& + \tfrac{1}{2}\sum _{i, a, b=1}^{n,m-n} V_{a} \cdot g(\nabla_{V_{a}}E_{i}, V_{b}) \ E_{i} \cdot V_{b} \cdot \widetilde{\psi} \qquad (V_{2}) \\
&+ \tfrac{1}{4}\sum _{a, b, c=1}^{m-n} V_{a} \cdot g(\nabla_{V_{a}}V_{b}, V_{c}) \ V_{b} \cdot V_{c} \cdot \widetilde{\psi} \qquad (V_{3}).
\end{align*}
Note that
\begin{align*}
\widetilde{D^{N}\psi}&=\xi_{\gamma}(\widetilde{D^{N}\psi}^1)\\
&= \xi_{\gamma}\left[(X_{i}^{*}\cdot X_{i}^{*}(\psi \circ \pi) +
g_1(\nabla^{1}_{X_{i}^{*}}X_{j}^{*}, X_{k}^{*}) X_{i}^{*} \cdot X_{j}^{*}
\cdot X_{k}^{*} \cdot (\psi \circ \pi)) \otimes \alpha\right],
\end{align*}
where $g_1(\nabla^{1}_{X_{i}^{*}}X_{j}^{*}, X_{k}^{*})=
h(\nabla^N _{X_{i}}X_{j}, X_{k})$.\\
The computation breaks down into five steps:

\textbf{Step 1}:
\begin{equation}\label{eq6}
\begin{split}
(H_{0})+(H_{1})+(H_{3})=&\lambda \widetilde{D^{N}\psi}
- \tfrac{n-1}{2} \mathrm{grad}^{\Hh}(\mathrm{ln}\lambda)\cdot \widetilde \psi
+ \sum _{i=1}^{n} E_{i} \cdot (\psi \circ \pi) \otimes \nabla^{\VV}_{E_{i}}\alpha.
\end{split}
\end{equation}
As
\begin{align*}
(H_{0})+(H_{1})=&\sum _{i=1}^{n} E_{i} \cdot \left[E_{i}(\psi \circ \pi) \otimes \alpha+ (\psi \circ \pi) \otimes E_{i}(\alpha)\right] \\
&+ \tfrac{1}{4}\sum _{i, j, k=1}^{n} E_{i} \cdot g(\nabla_{E_{i}}E_{j}, E_{k}) \ E_{j} \cdot E_{k} \cdot \widetilde{\psi},
\end{align*}
in order to recognize the lift of $D^N \psi$, first observe that
\begin{align*}
g(\nabla_{E_{i}}E_{j}, E_{k}) &=
g(\nabla_{\la X_{i}^{*}} \la X_{j}^{*}, \la X_{k}^{*}) \\
&= \lambda g_1(\nabla^{1}_{X_{i}^{*}}X_{j}^{*}, X_{k}^{*}) +
\left[X_{k}^{*}(\lambda)\delta_{i}^{j} - X_{j}^{*}(\lambda)\delta_{i}^{k} \right].
\end{align*}
Whilst
\begin{align*}
E_{i} \cdot \ E_{j} \cdot E_{k} \cdot \widetilde{\psi} &=
\xi_{\gamma}(E_{i}^1 \cdot \ E_{j}^1 \cdot E_{k}^1 \cdot \widetilde{\psi} ^1)\\
&= \xi_{\gamma}(X_{i}^{*} \cdot X_{j}^{*} \cdot X_{k}^{*} \cdot \widetilde{\psi} ^1),
\end{align*}
and
\begin{align*}
E_{i} \cdot E_{i}(\psi \circ \pi) \otimes \alpha &=
\la E_{i} \cdot E_{i}^1(\psi \circ \pi) \otimes \alpha \\
&= \la \xi_{\gamma}(X_{i}^{*}\cdot X_{i}^{*}(\psi \circ \pi) \otimes \alpha).
\end{align*}
Hence
\begin{align*}
(H_{0})+(H_{1}) =& \lambda \xi_{\gamma}(\widetilde{D^{N}\psi}^1)
+ \sum _{i=1}^{n} E_{i} \cdot (\psi \circ \pi) \otimes E_{i}(\alpha)  \\
&+ \xi_{\gamma}\left(\tfrac{1}{4}\sum _{i,j,k=1}^{n} \left[X_{k}^{*}(\lambda)\delta_{i}^{j} - X_{j}^{*}(\lambda)\delta_{i}^{k} \right] X_{i}^{*} \cdot X_{j}^{*} \cdot X_{k}^{*} \cdot \widetilde \psi ^1 \right).
\end{align*}
The last term can be rewritten
\begin{align*}
\tfrac{1}{4}\sum _{i,j,k=1}^{n} \left[X_{k}^{*}(\lambda)\delta_{i}^{j} - X_{j}^{*}(\lambda)\delta_{i}^{k} \right]
X_{i}^{*} \cdot X_{j}^{*} \cdot X_{k}^{*} \cdot \widetilde \psi ^1 =
-\tfrac{n-1}{2}\mathrm{grad}^{\Hh _1}(\lambda)\cdot \widetilde \psi ^1 ,
\end{align*}
and, since $\mathrm{grad}^{\Hh}(\lambda) =\la ^2 \mathrm{grad}^{\Hh _1}(\lambda)$, it becomes
\begin{align*}
-\tfrac{n-1}{2}\xi_{\gamma}\left(\mathrm{grad}^{\Hh _1}(\lambda)\cdot \widetilde \psi ^1\right) &=
-\tfrac{n-1}{2} \la \mathrm{grad}^{\Hh _1}(\lambda)\cdot \widetilde \psi \\
&= -\tfrac{n-1}{2} \mathrm{grad}^{\Hh}(\mathrm{ln}\lambda)\cdot \widetilde \psi.
\end{align*}
Summing up with ($H_{3}$), we obtain~\eqref{eq6}.

\textbf{Step 2}:
\begin{equation}
(V_{2})= -\tfrac{m-n}{2} \ \mu^{\VV} \cdot \widetilde{\psi} .
\end{equation}

As $g(\nabla_{V_{a}}E_{j}, V_{b})= - g(E_{j}, \nabla_{V_{a}} V_{b})$ and $\VV$ is integrable
\begin{align*}
(V_{2})&=-\tfrac{1}{2}\sum _{a, b=1}^{m-n} V_{a} \cdot (\nabla_{V_{a}} V_{b})^{\Hh} \cdot V_{b} \cdot \widetilde{\psi}\\
&=-\tfrac{1}{2}\left( \sum _{a < b=1}^{m-n} V_{a} \cdot (\nabla_{V_{a}} V_{b})^{\Hh} \cdot V_{b}
+ \sum _{a > b=1}^{m-n} V_{a} \cdot (\nabla_{V_{a}} V_{b})^{\Hh} \cdot V_{b} \right) \cdot \widetilde{\psi}
-\tfrac{1}{2} \sum _{a=1}^{m-n} (\nabla_{V_{a}} V_{a})^{\Hh} \cdot \widetilde{\psi}\\
&=-\tfrac{1}{2}\left( \sum _{a < b=1}^{m-n} V_{a} \cdot [V_{a}, V_{b}]^{\Hh} \cdot V_{b} \right) \cdot \widetilde{\psi}
-\tfrac{m-n}{2}  \mu^{\VV} \cdot \widetilde{\psi} \\
&= -\tfrac{m-n}{2} \ \mu^{\VV} \cdot \widetilde{\psi}.
\end{align*}

\textbf{Step 3}:
\begin{equation}
(V_{0})+(V_{3}) = (\ov \psi \circ \pi) \otimes D^{\VV} \alpha,
\end{equation}
since
\begin{align*}
(V_{0})+(V_{3})=&\sum _{a=1}^{m-n}  (\ov \psi \circ \pi) \otimes V_{a} \cdot V_{a}(\alpha)
+ \tfrac{1}{4}\sum _{a, b, c=1}^{m-n} V_{a} \cdot g(\nabla_{V_{a}}V_{b}, V_{c}) \ V_{b} \cdot V_{c} \cdot \widetilde{\psi} \\
=&(\ov \psi \circ \pi) \otimes V_{a} \cdot \nabla^{\VV}_{V_{a}}\alpha.
\end{align*}

\textbf{Step 4}:
\begin{equation}
(H_{2})= \tfrac{1}{2}I^{\Hh} \cdot \widetilde \psi - \tfrac{n}{2}\mu^{\Hh}\cdot \widetilde \psi.
\end{equation}
As for Step 2, we have
\begin{align*}
(H_{2})=\tfrac{1}{2}\sum _{i, j, a=1}^{n,m-n} E_{i} \cdot g(\nabla_{E_{i}}E_{j}, V_{a}) \ E_{j} \cdot V_{a} \cdot \widetilde{\psi}= \tfrac{1}{2}\sum _{i < j=1}^{n} E_{i} \cdot E_{j} \cdot I^{\Hh}(E_{i}, E_{j}) \cdot \widetilde{\psi} - \tfrac{n}{2}\mu^{\Hh}\cdot \widetilde \psi,
\end{align*}
where the terms $i=j$ give $\mu^{\Hh}$.

\textbf{Step 5}:
\begin{equation}
(V_{1})= - \tfrac{1}{4}I^{\Hh} \cdot \widetilde \psi,
\end{equation}
since
\begin{align*}
(V_{1})=& \tfrac{1}{4}\sum _{i, j, a=1}^{n,m-n} g(\nabla_{V_{a}}E_{i}, E_{j}) V_{a} \cdot \ E_{i} \cdot E_{j} \cdot \widetilde{\psi}   \\
&= \tfrac{1}{4}\sum _{i, j, a=1}^{n,m-n} \left[g([V_{a}, E_{i}], E_{j}) -
g(\nabla_{E_{i}}E_{j}, V_{a}) \right] V_{a} \cdot \ E_{i} \cdot E_{j} \cdot \widetilde{\psi}  \\
&= \tfrac{1}{4}\sum _{i, j, a=1}^{n,m-n} g([V_{a}, E_{i}], E_{j}) V_{a} \cdot \ E_{i} \cdot E_{j} \cdot \widetilde{\psi} - \tfrac{1}{2}(H_2).
\end{align*}
But $g([V_{a}, E_{i}], E_{j})=g([V_{a}, \la X_{i}^{*}], \la X_{j}^{*})
= \tfrac{V_{a}(\la)}{\la} \delta_i ^j$, since $[V_{a}, X_{i}^{*}]\in \VV$. \\
Therefore
\begin{align*}
(V_{1})=&-\tfrac{n}{4}\sum _{a=1}^{m-n} V_{a}(\mathrm{ln}\la) V_{a} \cdot \widetilde{\psi} - \tfrac{1}{2}(H_2)\\
&=- \tfrac{n}{4}\mathrm{grad}^{\VV}(\mathrm{ln}\la) \cdot \widetilde{\psi} - \tfrac{1}{2}(H_2)\\
&=- \tfrac{n}{4}\mu^{\Hh} \cdot \widetilde{\psi} - \tfrac{1}{2}(H_2) \\
&=- \tfrac{1}{4}I^{\Hh} \cdot \widetilde \psi.
\end{align*}
Summing up these five steps yields the chain rule.
\end{proof}

A generalization of \cite[Proposition 2.4]{A-G} to vector-valued functions yields local existence of harmonic spinors with prescribed value.

\begin{lm}[Local existence]\label{localexistence}
For any point $p\in M$ and $\psi_0 \in \SQ_p M$, there exists an open neighbourhood $U$ of $p$ and a harmonic spinor
$\psi : U \to \SQ M$ such that $\psi(p)=\psi_0$.
\end{lm}

\begin{te}[Characterization for horizontally conformal submersions]\label{th3.2}
Let $\pi: (M,g) \longrightarrow (N,h)$ be a horizontally conformal submersion between spin manifolds and 
assume there exists a section $\alpha$ satisfying the conditions of Definition~\ref{def4}.\\
Then $\pi$ is a Dirac morphism if and only if its horizontal distribution is integrable and
\begin{equation}\label{F-Eq}
(m-n) \mu^{\VV}+(n-1)\mathrm{grad}^{\Hh}(\mathrm{ln}\lambda) = 0,
\end{equation}
where $\mu^{\VV}$ is the mean curvature of the fibres.
\end{te}

\begin{proof}
Let $\pi$ be a horizontally conformal submersion with integrable horizontal distribution and such that \eqref{F-Eq} is satisfied.
In this case, the Chain rule \eqref{chainhwc} simplifies to
\begin{equation*}
D^{M} \widetilde{\psi}  =
\lambda \widetilde{D^{N}\psi}
+ \sum _{i=1}^{n} E_{i} \cdot (\psi \circ \pi) \otimes \nabla^{\VV}_{E_{i}}\alpha
+(\ov \psi \circ \pi) \otimes \left[D^{\VV}\alpha-\tfrac{n}{2} \mu^{\Hh} \cdot \alpha \right].
\end{equation*}
Let $\psi$ be a local harmonic spinor and $\alpha$ a section of $\Gamma(\SQ \VV)$, $\nabla^{\VV}$-parallel in horizontal directions and with 
$D^{\VV}\alpha - \frac{n}{2} \mu^{\Hh} \cdot \alpha= 0$. From the above formula, it follows that $D^{M} \widetilde{\psi}=0$ and therefore $\pi$ is a Dirac morphism.

Conversely, suppose that $\pi$ is a Dirac morphism. Then, by Definition~\ref{def4}, there exists a horizontally parallel section $\alpha \in \Gamma(\SQ \VV)$ 
satisfying $D^{\VV}\alpha - \frac{n}{2} \mu^{\Hh} \cdot \alpha= 0$, and, for a harmonic spinor field $\psi$ on $N$, we have, according to \eqref{chainhwc}
\begin{align} \label{Chain-Split}
0 = & - \tfrac{1}{2} \left[(m-n) \mu^{\VV}+(n-1)\mathrm{grad}^{\Hh}(\mathrm{ln}\lambda) \right] \cdot \widetilde{\psi} \\
& + \tfrac{1}{4}\sum_{i < j=1}^{n} X_{i}^{*} \cdot X_{j}^{*} \cdot (\psi \circ \pi) \otimes I^{\Hh}(X_{i}^{*} , X_{j}^{*}) \cdot \alpha \notag.
\end{align}
Putting $X=(m-n) \mu^{\VV}+(n-1)\mathrm{grad}^{\Hh}(\mathrm{ln}\lambda)$ and
$V^{ij}=I^{\Hh}(X_{i}^{*} , X_{j}^{*})$, \eqref{Chain-Split}  becomes
$$
0= - \tfrac{1}{2} X \cdot \widetilde{\psi}
+ \tfrac{1}{4}\sum _{i < j=1}^{n} X_{i}^{*} \cdot X_{j}^{*} \cdot V^{ij} \cdot \widetilde{\psi}.
$$
As the value of $\widetilde{\psi}$ at $p \in M$ can be prescribed, the above equation implies
$$
0= - \tfrac{1}{2} X + \tfrac{1}{4}\sum _{i < j=1}^{n} X_{i}^{*} \cdot X_{j}^{*} \cdot V^{ij}.
$$
But since $V^{ij}$ is vertical, $X$ and $X_{i}^{*} \cdot X_{j}^{*} \cdot V^{ij}$ have different degrees, necessarily
$X=0$ and $V^{ij}=0$. Therefore, if a horizontally conformal submersion is Dirac morphism, it must satisfy Equation~\eqref{F-Eq} and have integrable horizontal distribution.
\end{proof}

\begin{re}
\begin{enumerate}
\item Note the analogy between Equation~\eqref{F-Eq} and the fundamental equation for harmonic morphisms in~\cite{B-W}.
\item Compare Formula~\eqref{chainhwc} with \cite[(4.26)]{bis} and \cite[1.1.1]{moro}.\\
\item If the fibres are totally geodesic, the integrability of the horizontal distribution
makes the section $\alpha$ ``basic transversally harmonic'', as introduced
in~\cite{glaz}.
\end{enumerate}
\end{re}

\begin{co}
A Riemannian submersion $\pi : (M^m ,g) \to (N^n ,h)$ between spin manifolds is a Dirac morphism if and only if
its fibres are minimal and its horizontal distribution is integrable.
\end{co}

Recall that if $\pi$ is a Riemannian submersion then $\mu^{\Hh}=0$.

\begin{re}
If the dilation function $\lambda$ is a projectable function (i.e. $V(\lambda)=0$), the conformal invariance of the Dirac operator
(\cite{law}) allows a correspondence between harmonic spinors of the spaces involved in the commutative diagram below.
\begin{center}
\setlength{\unitlength}{1mm}
\begin{picture}(50,40)
\put(41,27){\vector(0,-3){27}}
\put(3,27){\vector(4,-3){35}}
\put(15,30){\vector(4,0){12}}
\put(8,-4){\vector(4,0){29}}
\put(0,27){\vector(0,-3){27}}
\put(42,33){\makebox(0,0)[t]{$(M, \pi^{*}h+g^{\VV})$}}
\put(0,33){\makebox(0,0)[t]{$(M,  \lambda^{2} \pi^{*}h+g^{\VV})$}}
\put(44,13){\makebox(0,0)[t]{$\pi$}}
\put(0,-2){\makebox(0,0)[t]{$(N, \tilde \lambda^{2}h)$}}
\put(20,0){\makebox(0,0)[t]{\bf 1$_N$}}
\put(20,34){\makebox(0,0)[t]{\bf 1$_M$}}
\put(2,13){\makebox(0,0)[t]{$\pi$}}
\put(43,-2){\makebox(0,0)[t]{$(N, h)$}}
\end{picture}
\end{center}
\end{re}

\bigskip

\section{Dirac morphisms with one-dimensional fibres}

In this section $m=n+1$.

\begin{de}
A horizontally conformal submersion $\pi : (M^{n+1}, g) \rightarrow (N^n, h)$ between spin manifolds is called a Dirac morphism if for any local harmonic spinor $\Psi$ defined on $U \subset N$, such that $\pi^{-1}(U) \neq \emptyset$, the pullback $\widetilde \Psi = \xi_{\gamma} \circ \widetilde \Psi _1$ is a harmonic spinor on $\pi^{-1}(U) \subset M$, where  $\widetilde \Psi _1$ is the pullback of $\Psi$ by the associated Riemannian submersion.
\end{de}

\begin{lm} \ \emph{(Chain rule)}. Let $\pi : (M^{n+1}, g) \rightarrow (N^n, h)$ be a horizontally conformal submersion of dilation $\lambda$ and $\psi$ a (local) spinor field on $N$, then
\begin{equation}\label{Chain-Split1}
\begin{split}
D^{M} \widetilde{\psi} =
&\lambda \widetilde{D^{N}\psi} -
\frac{1}{2} \left[\mu^{\VV}+(n-1)\mathrm{grad}^{\Hh}(\mathrm{ln}\lambda) +n \mu^{\Hh}\right] \cdot \widetilde \psi +
\frac{1}{4} I^{\Hh}\cdot \widetilde \psi.
\end{split}
\end{equation}
\end{lm}

\begin{proof} Take $\{X_i \}_{i=1,...,n}$
an orthonormal frame on $(N^n, h)$ and
$\{ V, E_i=\lambda X_i ^*\}_{i=1,...,n}$ an adapted frame on $(M^m, g)$. Let $\Psi$ be a (local) spinor field on $N$
and $\widetilde \Psi$ its pullback by $\pi$.
The proof is similar to the proof of Lemma 1, except that $(H_0)= \sum E_i \cdot E_i(\psi \circ \pi)$, $(V_0)=0$ and the terms $(H_3), (V_3)$ do not appear.
\end{proof}

Note that $\mu^{\Hh} \cdot \widetilde \psi =  i \norm{\mu^{\Hh}}^2 \widetilde{\ov{\psi}}$ as $V \cdot \widetilde \psi = i \widetilde{\ov \psi}$.

\begin{te}
A horizontally conformal submersion
$\pi : (M^{n+1} ,g) \to (N^n ,h)$ between spin manifolds is a Dirac morphism if and only if its
horizontal distribution is integrable and minimal, and
$$\mu^{\VV}+(n-1)\mathrm{grad}^{\Hh}(\mathrm{ln}\lambda)=0 ,$$
where $\mu^{\VV}$ is the mean curvature of the fibres.
\end{te}

\begin{proof}
The argument is similar to the one of Theorem 1, except that
$$
X=\mu ^{\VV} + (n - 1)\mathrm{grad}^{\Hh}(\ln \lambda)  + n \mu^{\Hh}.
$$
Observe that $X=0$ if and only if $\mu ^{\VV} + (n - 1)\mathrm{grad}^{\Hh}(\ln \lambda)=0$ and $\mu^{\Hh}=0$, as they belong to orthogonal distributions.
\end{proof}

\begin{co}
A Riemannian submersion $\pi : (M^{n+1}, g) \rightarrow (N^n, h)$ between spin manifolds is a Dirac morphism if and only if its horizontal distribution is integrable and the fibres are minimal.
\end{co}

\begin{re} Suppose that $\pi$ is a Riemannian submersion.

(1) \ The Chain Rule \eqref{Chain-Split1} gives us the formula of \cite{amoro} (where the fibres are minimal)

\begin{equation}\label{dirak1}
D^{M} \widetilde{\Psi} =
\widetilde{D^{N}\Psi} -\frac{1}{2} \ \mu^{\VV} \cdot \widetilde{\Psi}
-\frac{1}{4} i \sum_{j=1}^{n} X_{j}^{*} \cdot A_{X_{j}^{*}}V \cdot \widetilde{\ov{\Psi}},
\end{equation}
where $A$ denotes the second O'Neill tensor of $\pi$.

(2) \ If $\pi$ is a Dirac morphism, any pullback spinor field is parallel along the fibres and its \textit{metric Lie derivative} (\cite{bour}) vanishes.
\end{re}

\begin{re}
When the fibres are circles, a similar
chain rule was obtained by Ammann in~\cite{Am}.
\end{re}

\section{Examples}

\begin{ex}[Dirac morphisms from 3 to 2-dimensional Euclidean spaces]

With respect to the irreducible representation $\gamma$ of $\Cl_{2}$ given by Pauli matrices,
a spinor field on $(\RR^{2}, \langle, \rangle _{\mathrm{standard}})$ will be
$\psi:\RR^{2} \longrightarrow \CC^{2}$,
$\psi =\left(
\begin{array}{c}
\psi^{+} \\
\psi^{-}
\end{array}
\right)$, in a global spin frame ($e_{1}=\tfrac {\partial}{\partial y_{1}}, e_{2}=\tfrac {\partial}{\partial y_{2}}$) on $\RR^{2}$,
and the Dirac operator on $\RR^{2}$
$$
D^{\RR^{2}}=\gamma(e_{1})\frac {\partial}{\partial y_{1}}+\gamma(e_{2})\frac {\partial}{\partial y_{2}}
=\left(
\begin{array}{cc}
0 & -\tfrac{\partial}{\partial z} \\
\tfrac{\partial}{\partial \ov{z}} & 0
\end{array}
\right).
$$
Consider a Riemannian submersion $\pi:\RR^{3} \longrightarrow \RR^{2}$, the pull-back of $\psi$,
$\widetilde{\psi}=\psi \circ \pi:\RR^{3}\longrightarrow \CC^{2}$ is a spinor field on $\RR^{3}$
with respect to the frame $\{e_{1}^{*}, e_{2}^{*}, v \}$,
$v \in \Ker \dif \pi$. Supposing $\Hh$ integrable, this spin adapted frame can be chosen to be
$\{\tfrac {\partial}{\partial x_{1}}, \tfrac {\partial}{\partial x_{2}}, \tfrac {\partial}{\partial x_{3}} \}$, and
with the Pauli representation, the Dirac operator on $\RR^{3}$ is
$$
D^{\RR^{3}}=i\sigma_{k} \frac {\partial}{\partial x_{k}}
=\left(
\begin{array}{cc}
i \tfrac {\partial}{\partial x_{3}} & \tfrac {\partial}{\partial x_{2}}+i\tfrac {\partial}{\partial x_{1}} \\
-\tfrac {\partial}{\partial x_{2}}+i\tfrac {\partial}{\partial x_{1}} & - i \tfrac {\partial}{\partial x_{3}}
\end{array}
\right).
$$
Let $\psi$ be an arbitrary harmonic spinor on $\RR^{2}$ (so $\psi^{+}$ a holomorphic function), then
$\widetilde{\psi}^{+}$ is also harmonic if and only if
$$
\frac{\partial \pi_{1}}{\partial x_{1}} =
\frac{\partial \pi_{2}}{\partial x_{2}},
\quad
\frac{\partial \pi_{1}}{\partial x_{2}} =
-\frac{\partial \pi_{2}}{\partial x_{1}},
\quad
\frac{\partial \pi_{1}}{\partial x_{3}} =
\frac{\partial \pi_{2}}{\partial x_{3}} = 0.
$$
The analogous question for $\widetilde{\psi}^{-}$ involves a change of sign in the first two equalities
(i.e. $\pi$ must be anti-holomorphic with respect to $x_{1}+ ix_{2}$).

These conditions are exactly the harmonicity of $\pi$, which is equivalent to the minimality of the fibres.

\end{ex}

\begin{ex}[Dirac morphisms from 4 to 2-dimensional Euclidean spaces]

A spinor field on $(\RR^{4}, \langle, \rangle _{\mathrm{standard}})$ can be seen as a $\CC^{4}$-valued function on $\RR^{4}$ once a global spin frame is chosen. Using Pauli matrices, the Dirac operator on $\RR^{4}$ can be described as
$$
D^{\RR^{4}}=\gamma_{k}  \frac {\partial}{\partial x_{k}}
=\left(
\begin{array}{cccc}
0 & 0 & \tfrac {\partial}{\partial x_{3}}+i \tfrac {\partial}{\partial x_{0}} &
\tfrac {\partial}{\partial x_{1}}-i\tfrac {\partial}{\partial x_{2}} \\
0 & 0 & \tfrac {\partial}{\partial x_{1}}+i\tfrac {\partial}{\partial x_{2}} &
-\tfrac {\partial}{\partial x_{3}}+i \tfrac {\partial}{\partial x_{0}} \\
-\tfrac {\partial}{\partial x_{3}}+i \tfrac {\partial}{\partial x_{0}} &
-\tfrac {\partial}{\partial x_{1}}+i\tfrac {\partial}{\partial x_{2}} & 0 & 0\\
-\tfrac {\partial}{\partial x_{1}}-i\tfrac {\partial}{\partial x_{2}} &
\tfrac {\partial}{\partial x_{3}}+i \tfrac {\partial}{\partial x_{0}}  & 0 & 0
\end{array}
\right).
$$
Let $\pi:\RR^{4} \longrightarrow \RR^{2}$ be a Riemannian submersion and $\psi: \RR^{2} \longrightarrow \CC^{2}$
 a spinor field on $(\RR^{2}, \langle, \rangle _{\mathrm standard})$. Then the pull-back $\widetilde{\psi}$ of $\psi$ is
$(\psi \circ \pi) \otimes \alpha :\RR^{4}\longrightarrow \CC^{4}$. With respect to $\{e_{1}^{*},e_{2}^{*},v,w\}$,
$v, w \in \Ker \dif \pi$, assuming $\Hh$ integrable and choosing a frame
$\{\tfrac {\partial}{\partial x_{2}},\tfrac {\partial}{\partial x_{3}}, \tfrac {\partial}{\partial x_{0}},
\tfrac {\partial}{\partial x_{1}} \}$,
$$
\widetilde{\psi}
=\left(
\begin{array}{c}
\widetilde{\psi}^{+} \\
\widetilde{\psi}^{-}
\end{array}
\right)
=\left(
\begin{array}{c}
\psi^{+} \alpha^{+} \\
\psi^{-} \alpha^{-} \\
\psi^{+} \alpha^{-} \\
\psi^{-} \alpha^{+}
\end{array}
\right).
$$
The conditions of harmonicity and parallelism make $\alpha: \RR^{2} \rightarrow \CC^{2}$
a harmonic spinor field with respect to the variables $x_{0}, x_{1}$ (i.e. $\alpha^{+}$ is holomorphic
and $\alpha^{-}$ anti-holomorphic).
Take a harmonic spinor $\psi$ on $\RR^{2}$
(i.e. $\psi^{+}$ is a holomorphic function), one can directly check that $\widetilde{\psi}^{+}$
is harmonic, for any $\alpha$, if and only if $\pi$ satisfies
$$
\frac{\partial \pi_{1}}{\partial x_{0}} =\frac{\partial \pi_{1}}{\partial x_{1}} =
\frac{\partial \pi_{2}}{\partial x_{0}}=\frac{\partial \pi_{2}}{\partial x_{1}}= 0,
\quad \frac{\partial \pi_{1}}{\partial x_{2}} =
\frac{\partial \pi_{2}}{\partial x_{3}}, \quad
\frac{\partial \pi_{1}}{\partial x_{3}} =
-\frac{\partial \pi_{2}}{\partial x_{2}}.
$$

The same question for $\widetilde{\psi}^{-}$ merely introduces a different sign. Again, this forces $\pi$ to be harmonic, i.e. its fibres are minimal.

\end{ex}

\begin{ex}[Moroianu's projectable spinors]\label{exmoro}

In \cite{moro}, Moroianu considers a principal fibre bundle $\pi: (M^{m},g)
\longrightarrow N$ with compact structural group $G$, over a compact spin manifold
$(N^{n},h)$, such that $\pi$ is a Riemannian submersion with totally geodesic fibres
and the horizontal distribution $\Hh$ is a principal connection.

Since its tangent space is trivial, $G$ admits a canonical spin structure, and a
spinor $(\psi \circ \pi) \otimes \alpha$ is called projectable if $\alpha: G
\longrightarrow \SQ_{m-n}$ is a constant function with respect to the canonical frame
of left-invariant vector fields. To have a $D^M$-invariant notion of projectable
spinor, it is necessary and sufficient to suppose $G$ commutative (\cite{moro})

Let $X^{*}$ be the horizontal lift of a vector field $X$ on $N$ and, using $[V,
X^{*}]=0$ for $V\in \VV$, since $\Hh$ is a principal connection, we have
\begin{equation*}
\begin{split}
\nabla^{\VV}_{X^{*}}\alpha=& X(\alpha)+ \tfrac{1}{2}\sum _{b<c=1}^{m-n} g(\nabla_{X^{*}}V_{b}, V_{c})V_{b} \cdot V_{c} \cdot \alpha\\
=&X^{*}(\alpha)+ \tfrac{1}{2}\sum _{b<c=1}^{m-n} g(\nabla_{V_{b}}X^{*}, V_{c})V_{b} \cdot V_{c} \cdot \alpha\\
=&X^{*}(\alpha)- \tfrac{1}{2}\sum _{b<c=1}^{m-n} g(X^{*}, \nabla_{V_{b}}V_{c})V_{b} \cdot V_{c} \cdot \alpha\\
=&X^{*}(\alpha),
\end{split}
\end{equation*}
since the fibres are totally geodesic. Therefore the condition of parallelism of $\alpha$ translates into constancy in horizontal directions. Moreover

\begin{equation*}
\begin{split}
&D^{\VV}\alpha = V_{a} \cdot \nabla^{\VV}_{V_{a}}\alpha \\
&=\sum _{a=1}^{m-n}  V_{a} \cdot V_{a}(\alpha)
+ \tfrac{1}{4}\sum _{a, b, c=1}^{m-n} V_{a} \cdot g(\nabla_{V_{a}}V_{b}, V_{c}) \ V_{b} \cdot V_{c} \cdot \alpha \\
&= \sum _{a=1}^{m-n}  V_{a} \cdot V_{a}(\alpha)
+ \tfrac{3}{4}\sum _{a < b < c=1}^{m-n} g([V_{a},V_{b}], V_{c}) \ V_{a} \cdot \ V_{b} \cdot V_{c} \cdot \alpha, \\
\end{split}
\end{equation*}
where, on a Lie group, $g([V,W], Z)=g(V, [W,Z])$ if $g$ is an invariant metric.
Clearly, if $\alpha$ is a constant function and $G$ is commutative (i.e. the
structural constants vanish) then $D^{\VV}\alpha =0$.

Therefore a projectable spinor field is the pull-back of some spinor field on the
base, with respect to $\alpha$ satisfying the two conditions required for Dirac
morphisms. Hence, Theorem~\ref{th3.2} says that a principal bundle over a spin
manifold with commutative structural group is a Dirac morphism if and only if it is
flat.
\end{ex}


\begin{thebibliography}{99}

\bibitem{A-G}
Alinhac, S., G\'erard, P.: \textit{Op\'erateurs pseudo-diff\'erentiels et th\'eor\`eme de Nash-Moser}. InterEditions, Paris; \'Editions du CNRS, Meudon, 1991

\bibitem{Am}
Ammann, B.: The Dirac operator on collapsing $S\sp 1$-bundles.
S\'emin. Th\'eor. Spectr. G\'eom. \textbf{16}, Univ. Grenoble I,
Saint-Martin-d'H\`eres, 1998

\bibitem{B-W}
Baird, P., Wood, J.C.: \textit{Harmonic morphisms between Riemannian manifolds}. Oxford Univ. Press, 2003

\bibitem {bar} B\"ar, C., Gauduchon, P., Moroianu, A.: Generalized cylinders in semi-Riemannian and Spin geometry.
Math. Zeit. {\bf 249}, 545--580 (2005)

\bibitem {bis} Bismut, J.-M., Cheeger, J.: $\eta$-invariants and their adiabatic limits.
J. Amer. Math. Soc. {\bf 2}, 33--70 (1989)

\bibitem {bour} Bourguignon, J.-P., Gauduchon P.: Spineurs, op{\'e}rateurs de Dirac et variations de m{\'e}triques.
Comm. Math. Phys. {\bf 144}, 581--599 (1992)

\bibitem{Fuglede} Fuglede, B.: Harmonic morphisms between Riemannian manifolds.
Ann. Inst. Fourier (Grenoble) {\bf 28}, 107--144 (1978)

\bibitem {glaz} Glazebrook, J.F., Kamber, F.: Transversal Dirac families in Riemannian foliations.
Comm. Math. Phys. {\bf 140}, 217--240 (1991)


\bibitem {hus} Husemoller, D.: \textit{Fibre Bundles}.
Springer-Verlag, 1994

\bibitem{Ishihara} Ishihara, T.: A mapping of Riemannian manifolds which  preserves harmonic functions.
J. Math. Kyoto Univ. {\bf 19}, 215--229 (1979)

\bibitem{Jacobi} Jacobi, C.G.J.:
{\"U}ber eine L{\"o}sung der partiellen Differentialgleichung $\Delta(V)=0$. J. Reine Angew. Math. {\bf 36}, 113--134 (1848)

\bibitem {law} Lawson, H., Michelsohn M.-L.: \textit{Spin Geometry}. Princeton University Press, 1989

\bibitem {moro} Moroianu, A.: \textit{Op{\'e}rateur de Dirac et submersions riemanniennes}. Th{\`e}se de Doctorat, Ecole Polytechnique,
1996

\bibitem{amoro} Moroianu, A.: La premi{\`e}re valeur propre de l'op{\'e}rateur de Dirac sur les vari{\'e}t{\'e}s k\"ahl{\'e}riennes
compactes. Comm. Math. Phys. {\bf 169},  373--384 (1995)


\end{thebibliography}
\end{document}